\documentclass[14pt]{amsart}

\usepackage{pifont}
\usepackage{mathrsfs}
\usepackage{amscd,color}
\usepackage{amsmath}
\usepackage{latexsym}
\usepackage{amsfonts}
\usepackage{amssymb}
\usepackage{amsthm}
\usepackage{graphicx}
\usepackage{amsmath,amsfonts}
\usepackage{CJK}
\usepackage[linkcolor=blue,citecolor=blue]{hyperref}

\makeatletter
\newcommand{\Extend}[5]{\ext@arrow0099{\arrowfill@#1#2#3}{#4}{#5}}
\makeatother
\let\pa=\partial

\def\eps\epsilon

\def\cA{{\cal A}}

\def\C{\mathop{\bf C\kern 0pt}\nolimits}
\def\DD{\mathop{\bf D\kern 0pt}\nolimits}
\def\K{\mathop{\bf K\kern 0pt}\nolimits}
\def\N{\mathop{\bf N\kern 0pt}\nolimits}
\def\Q{\mathop{\bf Q\kern 0pt}\nolimits}

\newcommand{\beq}{\begin{equation}}
\newcommand{\eeq}{\end{equation}}
\newcommand{\ben}{\begin{eqnarray}}
\newcommand{\een}{\end{eqnarray}}
\newcommand{\beno}{\begin{eqnarray*}}
\newcommand{\eeno}{\end{eqnarray*}}

\def\R{\mathop{\mathbb R\kern 0pt}\nolimits}

\newtheorem{theorem}{Theorem}[section]
\newtheorem{proposition}[theorem]{Proposition}
\newtheorem{lemma}[theorem]{Lemma}
\newtheorem{remark}{Remark}[section]

\theoremstyle{remark}

\begin{document}

 \title[Scattering theory for focusing $2d$ NLS with potential]{ \bf Scattering theory for NLS with inverse-square potential in 2D }

 \author{Xiaofen Gao}%
\address{Department of Mathematics, Beijing Institute of Technology, Beijing,China, 100081}
\email{xiaofengao\_am@163.com}

\author{Chengbin Xu}%
\address{The Graduate School of China Academy of Engineering Physics, P. O. Box 2101, Beijing, China, 100088}
\email{xcbsph@163.com}

\maketitle

\begin{abstract}
  In this paper, we study the long time behavior of the solution of nonlinear Schr\"odinger equation with a singular potential.   We prove scattering below the ground state for the radial NLS with inverse-square potential in dimension two
  $$iu_t+\Delta u-\frac{au}{|x|^2}=- |u|^pu$$
  when $2<p<\infty$ and $a>0$. This work extends the result in \cite{KMVZ,LMM,ZZ1} to dimension 2D. The key point is a modified version of Arora-Dodson-Murphy's approach \cite{ADM}.
\end{abstract}

\begin{center}
 \begin{minipage}{12.2cm}
 \end{minipage}
 \end{center}

\begin{CJK*}{GBK}{song}
\setcounter{section}{0}\setcounter{equation}{0}
\section{Introduction}

\noindent

 Considering the initial value problem (IVP), also called the Cauchy problem for the nonlinear Schr\"odinger equation with inverse potential
 \begin{align}\label{INLS}
 \begin{cases}
   &i\pa_tu+\Delta u-V(x)u=\lambda|u|^{p}u,\ \ \ \  t\in\R,\ x\in\R^2\\
   &u(0,x)=u_0(x) \in  H_a^1(\R^2)
 \end{cases}
 \end{align}
where $u:\R_t\times\R_x^2\rightarrow \mathbb{C} $ and $V(x)=\frac{a}{|x|^2}$ with $a>0.$ And $\lambda\in\{1,-1\}$ with $\lambda=1$ known as defocusing case and $\lambda=-1$ as the focusing case.

We consider the operator
$$\mathcal{L}_a:=-\Delta+\frac{a}{|x|^2}, a\geq -(d-2)^2/4.$$
More precisely, we interpret $\mathcal{L}_a$ as the Friedrichs extension of this
operator defined initially on $C_c^\infty(\R^d\setminus\{0\})$.  The
restriction $a\geq -(d-2)^2/4$ ensures that the operator $\mathcal{L}_a$ is positive semi-definite, see \cite[Sect. 1.1]{KMVZZ1}.

The nonlinear Schr\"odinger equations with inverse-square potential have attracted a lot of interest in the past years (see e.g.\cite{BPST, KMVZZ, KMVZZ1, KMVZ, LMM, ZZ1, ZZ, Zheng}).
However, to the best of our knowledge, there is little result for the low dimension $d=2$. In this paper, we will prove scattering theory for case $d=2$.

When $d=2$, restricting naturally to values $a>0$, we consider $\mathcal{L}_a$ as the Friedrich extension of the quadratic form $Q$,(see \cite{Fa-2013} and \cite{FFT}), defined on $C_c^\infty(\R^2\setminus\{0\})$ via
$$Q(f):=\int_{\R^2}|\nabla f|^2+a\frac{|f|^2}{|x|^2}dx.$$
Let us define $\dot{H}_a^1(\R^2)$ as the completion of $C_c^\infty(\R^2\setminus\{0\})$ with respect to the norm
$$\|u\|_{\dot{H}_a^1(\R^2)}:=\left(\int_{\R^2}|\nabla f|^2+a\frac{|f|^2}{|x|^2}dx\right)^{\frac12}.$$



In dimension $d\geq3$, Killip-Miao-Visan-Zhang-Zheng \cite{KMVZZ1} established harmonic analysis for $\mathcal{L}_a$ and showed the Sobolev norm properties. But for $d=2$, the above properties fails.
And by Hardy inequality,  the spcace $H_a^1$ is strictly smaller than the classical Sobolev space $H^1$, hence we can not use the chain rule in \cite{KMVZZ1} which is a crucial step to show the well-posedness. Instead,  we will use Aharonov-Bohn potential \cite{LW} to define the Sobolev norm $\|\cdot\|_{{H}_A}$ and prove the norm of $H_a^1$ coincide with $H_A^1$ for radial function, thus we can use the quadratic form of the norm of $\dot{H}_A$ to
obtain the chain rule. Therefore we can prove the well posedness and scattering theory of \eqref{INLS}.

The class of solutions to \eqref{INLS} 
conserve their mass and energy, defined respectively by
$$M(u)=:\int_{\R^2}|u|^2dx=M(u_0)$$
$$E(u)=:\int_{\R^2}\frac12(|\nabla u|^2+V|u|^2)-\frac1{p+2}|u|^{p+2}dx=E(u_0)$$
Initial data belonging to $H_a^1(\R^2)$ have finite mass and energy and the following variant of the Gagliardo-Nirenberg inequality:
$$\|f\|_{L_x^{p+2}}^{p+2}\leq C_a\|f\|_{L_x^2}^{2}\|f\|_{\dot{H}_a^1}^{p}$$
where $C_a$ denotes the sharp constant in the inequality above the radial functions. The sharp constant $C_a$ is attained by a radial solution $Q_a$ (see \cite{Zheng}), to elliptic equation
$$-\mathcal{L}_a-Q_a+Q_a^{p+1}=0.$$

In this paper, we consider the mass-supercritical range $2<p<\infty$, in two spatial dimension.
 We will give simple proof of the following scattering results.

\begin{theorem}[Radial Scattering]\label{Main}
Let $2<p<\infty$, $\lambda=-1$ and $a>0$. Suppose $u_0\in H_a^1(\R^2)$ is radial and $M(u_0)^{1-s_p}E(u_0)^{s_p}<M(Q_a)^{1-s_p}E(Q_a)^{s_p}$.
  Moreover, If
  $$\|u_0\|_{L^2}^{1-s_p}\|u_0\|_{\dot{H}_a^1}^{s_p}\leq\|Q_a\|_{L^2}^{1-s_p}\|Q_a\|_{\dot{H}_a^1}^{s_p}.$$
  Then the solution to \eqref{INLS} with data $u_0$ is global and scatters.
\end{theorem}
\begin{remark}
  In the case without potential (i.e. $a=0$), Theorem \ref{Main} result was previously established in \cite{AN,CF,G}. In these works the authors proved via the concentration compactness. Recently, Arora-Dodson-Murphy \cite{ADM} give a simple proof with radial initial data,  which avoids concentration compactness.

\end{remark}


\begin{remark}
The method here also can be use to treat the defocusing cases with $a>0$.  In the defocusing case, we can prove the scattering theory with the condition $a>\frac14$ via interaction Morawetz estimate by following \cite{ZZ1}, see Appendix.
In this sense, we extends the results in \cite{ZZ1} for $a>0$ under the radial assumption.
\end{remark}

The rest of this paper is organized as follows: In section 2, we set up some notation, recall some important theory for the $\mathcal{L}_a$.  In section 3, we establish a new scattering criterion for \eqref{INLS}, Lemma \ref{S-C}. In section 4, by the Morawetz identity, we will establish the virial/Morawetz estimates to show the solution satisfy the scattering criterion of Lemma \ref{S-C}, thereby completing the proof of Theorem \ref{Main}. In appendix, we will establish interaction Moraweta type estimate for the nonradial defocusing case, then we can obtain global solution scattering.

We conclude the introduction by giving some notations which
will be used throughout this paper. We always use $X\lesssim Y$ to denote $X\leq CY$ for some constant $C>0$.
Similarly, $X\lesssim_{u} Y$ indicates there exists a constant $C:=C(u)$ depending on $u$ such that $X\leq C(u)Y$.
We also use the big-oh notation $\mathcal{O}$. e.g. $A=\mathcal{O}(B)$ indicates $C_{1}B\leq A\leq C_{2}B$ for some constants $C_{1},C_{2}>0$.
The derivative operator $\nabla$ refers to the spatial  variable only.
We use $L^r(\mathbb{R}^2)$ to denote the Banach space of functions $f:\mathbb{R}^2\rightarrow\mathbb{C}$ whose norm
$$\|f\|_r:=\|f\|_{L^r}=\Big(\int_{\mathbb{R}^2}|f(x)|^r dx\Big)^{\frac1r}$$
is finite, with the usual modifications when $r=\infty$. For any non-negative integer $k$,
we denote by $H^{k,r}(\mathbb{R}^2)$ the Sobolev space defined as the closure of smooth compactly supported functions in the norm $\|f\|_{H^{k,r}}=\sum_{|\alpha|\leq k}\|\frac{\partial^{\alpha}f}{\partial x^{\alpha}}\|_r$, and we denote it by $H^k$ when $r=2$.
For a time slab $I$, we use $L_t^q(I;L_x^r(\mathbb{R}^2))$ to denote the space-time norm
\begin{align*}
  \|f\|_{L_{t}^qL^r_x(I\times \R^2)}=\bigg(\int_{I}\|f(t,x)\|_{L^r_x}^q dt\bigg)^\frac{1}{q}
\end{align*}
with the usual modifications when $q$ or $r$ is infinite, sometimes we use $\|f\|_{L^q(I;L^r)}$ or $\|f\|_{L^qL^r(I\times\mathbb{R}^2)}$ for short.

\section{Preliminaries}
\noindent

In this section, we first introduce the Sobolev norm associated with Aharonov-Bohn potential and show the equivalence. Next we recall the dispersive estimates and Strichartz estimates.

\subsection{The equivalent Sobolev norm}

\noindent

Define
$$\mathcal{L}_A:=\left(-i\nabla+A(x)\right)^2\ \ with\ \ A(x)=\alpha\left(-\frac{x_2}{|x|^2},\frac{x_1}{|x|^2}\right)$$
formally acts on function $f$ as
$$\mathcal{L}_A=-\Delta f+\frac{\alpha^2}{|x|^2}f-2i\alpha\left(-\frac{x_2}{|x|^2},\frac{x_1}{|x|^2}\right)\cdot\nabla f.$$
Let $\nabla_A:=\nabla+iA(x),$ define $\dot{H}_A^1$ as the completion of $C_c^\infty(\R^2\setminus\{0\})$ with respect to the norm
$$\|f\|_{\dot{H}_A^1}:=\left(\int_{\R^2}|\nabla_Af(x)|^2dx\right)^{\frac12}$$
In \cite{LW}, the quadratic form $\dot{H}_A$ can be written in a more convenient form by using the polar coordinate
$$h(\textbf{a})[u]=\int_{0}^{\infty}\int_{0}^{2\pi}(|u_r|^2+r^{-2}|u_{\theta}+i\alpha u|^2)rdrd\theta.$$
Thus, for any radial $u\in \dot{H}_A$, we have
$$h(\textbf{a})[u]=2\pi\int_{0}^{\infty}(|u_r|^2+\alpha^2\frac{|u|^2}{r^2})rdr$$
and
$$\int_{\R^2}\frac{|u|^2}{|x|^2}dx=2\pi\int_{0}^{\infty}\frac{|u|^2}{r}dr\leq \frac{1}{\alpha^2}h(\textbf{a})[u].$$
Hence, if $u$ is radial,  we obtain that
\begin{equation}
\|u\|^2_{\dot{H}_a^1(\R^2)}=\int_{\R^2}\big(|\partial_r f|^2+a\frac{|f|^2}{r^2}\big) rdr\lesssim h(\textbf{a})[u]=\|u\|^2_{\dot{H}_A^1}.
\end{equation}
On the other hand, we have
\begin{equation}
\|u\|^2_{\dot{H}_A^1}=h(\textbf{a})[u]\lesssim \int_{\R^2}\big(|\partial_r f|^2+a\frac{|f|^2}{r^2}\big) rdr\lesssim \|u\|^2_{\dot{H}_a^1(\R^2)}, \qquad a=\alpha^2.
\end{equation}
So, it follows the equivalent norm  when $\alpha=\sqrt{a}(a>0)$
\begin{equation}
\|u\|_{\dot{H}_A^1}\cong\|u\|_{\dot{H}_a^1},\qquad u~\text{is radial}.
\end{equation}
Define the inhomogenous space $H_A^1=\dot{H}_A^1\cap L^2(\R^2)$, we obtain  $H_a^1$ coincide with $H_A^1.$ Moreover, for the radial $f$ and $\alpha^2=a$, we have $\mathcal{L}_a f=\mathcal{L}_A f. $ Then we may deduce that
$$[\nabla_A, \mathcal{L}_a]f=[\nabla_A, \mathcal{L}_A]f=0,\ \ \text{for} \ f \ \text{ radial}.$$

\subsection{Dispersive and Strichartz estimates}

\noindent

We recall some Strichartz estimates associated to the linear Schr\"odinger propagator in the radial case.

We say the pair $(q,r)$ is $L^2-$adimissible or simply admissible pair if they satisfy the condition
$$\frac{2}{q}+\frac{2}{r}=1$$
where $2\leq q,r\leq \infty.$
Let $\Lambda_0=\{(q,r):(q,r) \ is\  L^2-admissible \}$. 



\begin{lemma}[Dispersive estimate \cite{Fa-2013,Fa-2015}] Let $a>0$ and $2\leq p\leq \infty$, then we have
$$\|e^{it\mathcal{L}_a}u_0\|_{L^{p}}\leq C|t|^{-2(\frac12-\frac1p)}\|u_0\|_{L^{p'}}$$
for some constant $C=C(a,p)>0$ which does not depend on $t,\ u_0.$
\end{lemma}
\begin{proposition}[Strichartz estimates \cite{KT}]
Let $a>0$ and $\alpha=\sqrt{a}$. Suppose $u:I\times\R^2\rightarrow\mathcal{C}$ is a solution to $i\pa_tu-\mathcal{L}_au=F$ with initial data $u(t_0)$. Then for any $(q,r),\ (m,n)\in\Lambda_0$, we have
\begin{align}\label{S1}
  \|u\|_{L_t^qL_x^r(I\times\R^2)}\lesssim\|u(t_0)\|_{L^2}+\|F\|_{L_t^{m'}L_x^{n'}(I\times\R^2)}
\end{align}
and for radial solution
\begin{align}\label{S2}
\|\nabla_Au\|_{L_t^qL_x^r(I\times\R^2)}\lesssim\|\nabla_Au(t_0)\|_{L^2}+\|\nabla_AF\|_{L_t^{m'}L_x^{n'}(I\times\R^2)}
\end{align}
 \end{proposition}
\begin{proof}
  The first estimate \eqref{S1} is a direct consequence  of the argument of Keel-Tao \cite{KT} and the above dispersive estimate. By the fact $[\nabla_A,\mathcal{L}_a]f=0$ with radial $f$, we obtain \eqref{S2} from \eqref{S1}.
\end{proof}



\section{Local and global well posedness}

\noindent

As a consequence of the Strichartz estimate, we obtain the local well-posedness theory in $H^1_a(\R^2)$.

\begin{theorem}[Local well posedness]
Let $\alpha=\sqrt{a}$ and $ a>0$. Assume $u_0\in H_a^1$ is radial. For $t_0\in\R$, then there exists $T=T(\|u_0\|_{H_a^1})>0$ and a unique solution $u:(-T,T)\times \R^d\rightarrow\mathbb{C}$ to \eqref{INLS} with $u(t_0)=u_0$ satisfies
$$u,\ \nabla_A u\in C(I;L^2(\R^2))\cap L_t^q(I;L^r(\R^2)),\ \ \ I=[a,b]\subset(-T,T)$$
 In particular, if $u$ remains uniformly bounded in $H_a^1$ throughout its lifespan, then $u$ extends to a global solution.
\end{theorem}
\begin{proof}
  Using the equivalent of $\dot{H}_A^1$ and $\dot{H}_a^1$ , we will prove as in \cite{ZZ}. We define solution map
  $$\Phi:u\mapsto e^{it\mathcal{L}_a}u_0-i\int_{0}^{t}e^{i(t-s)\mathcal{L}_a}|u(s)|^pu(s)ds$$
on the complete metric space
$$B:=\left\{u,\nabla_A\in L_t^\infty(I,L_x^2):\|u\|_{L_t^\infty L_x^2},\|\nabla_Au\|_{L_t^\infty L_x^2}\leq 2C\|u_0\|_{H_a^1}\right\}$$
and the metric
$$d(u,v):=\|u-v\|_{L_t^\infty L_x^2(I\times \R^2)}$$
The constant $C$ depends only on the dimension and $p$, and it reflects implicit constants in the Strichartz and Sobolev embedding inequalities. We need prove that the operator $\Phi$ is well-defined on $B$ and is a contraction map under the metric $d$ for $I$.

Throughout the proof, all spacetime norms will be on $I\times\R^2$. By Strichatz inequality and Sobolev embedding, we have
\begin{align*}
\|\Phi(u)\|_{L_t^\infty L_x^2}\leq&C\|u_0\|_{L_x^2}+C\||u|^pu\|_{L_{t,x}^{\frac43}}\\
\leq& C\|u_0\|_{L_x^2}+CT^{\frac34}(\|u\|_{L_t^\infty L_x^2}\|u\|_{L_t^\infty L_x^{4p}}^p)\\
 \leq& C\|u_0\|_{L_x^2}+CT^{\frac34}(2C\|u_0\|_{H_a^1})^p
\end{align*}
Similarly, by Strichartz, we get
\begin{align*}
\|\nabla_A\Phi(u)\|_{L_t^\infty L_x^2}\leq&C\|\nabla_A u_0\|_{L_x^2}+C\|\nabla(|u|^pu)\|_{L_{t,x}^{\frac43}}+\||x|^{-1}|u|^{p+1}\|_{L_{t,x}^{\frac43}}\\
\leq& C\|u_0\|_{L_x^2}+CT^{\frac34}(\|\nabla u\|_{L_t^\infty L_x^2}\| u\|_{L_t^\infty L_x^{4p}}^p)+CT^{\frac34}(\||x|^{-1}u\|_{L_t^\infty L_x^2}\| u\|_{L_t^\infty L_x^{4p}}^p)\\
 \leq& C\|\nabla_A u_0\|_{L_x^2}+CT^{\frac34}(2C\|u_0\|_{H_a^1})^p
\end{align*}
Taking $T$ sufficiently small such that
$$T^{\frac34}(2C\|u_0\|_{H_a^1})^p\leq \|u_0\|_{H_a^1}$$
Thus $\Phi $ maps $B$ to itself.

Finally, for $u,v\in B$, we argument as above
$$d(\Phi(u),\Phi(v))\leq 2CT(2C\|u_0\|_{H_a^1})^{p-1}d(u,v)\leq \frac12 d(u,v)$$
by taking $T$ sufficiently small such that
$$2CT(2C\|u_0\|_{H_a^1})^{p-1}\leq\frac12$$

The standard fixed point argument gives a unique solution $u$ of \eqref{INLS} on $I\times \R^2.$ We also have
$$\|u\|_{L_t^qL_x^r}+\|\nabla_Au\|_{L_t^qL_x^r}\leq 2C\|u_0\|_{H_a^1}$$
The $T$ only depends $\|u_0\|_{H_a^1},\ p$ and $C$, if $\|u(t)\|_{H_a^1}$ is uniformly bounded, then $u(t)$ is global.
\end{proof}
\begin{lemma}\cite{Zheng}
  Fix $a>0$ and define
  $$C_a:=\sup\{\|f\|_{L^{p+2}}^{p+2}\div[\|f\|_{L_x^2}^{2}\|f\|_{\dot{H}_a^1}^{p}]:f\in H_a^1\setminus\{0\},\ f\  \text{radial}\}$$
  Then $C_a\in(0,\infty)$ and the Gagliardo-Nirenberg inequality for radial functions
  $$\|f\|_{L^{p+2}}^{p+2}\leq C_a\|f\|_{L_x^2}^{2}\|f\|_{\dot{H}_a^1}^{p}$$
  is attained by a function $Q_a\in H_a^1,$ which is non-zero, non-negative, radial solution to elliptic problem
  $$-\mathcal{L}_aQ_a-Q_a+Q_a^{p+1}=0.$$
\end{lemma}

\begin{lemma}[Coercivity, \cite{Zheng}]\label{V-C}
 Fix $a>0$. Let $u: I\times\R^2\rightarrow\mathbb{C}$ be the maximal-lifespan solution to \eqref{INLS} with $u_0\in H_a^1$. Assume that
$$M(u_0)^{1-s_p}E(u_0)^{s_p}\leq (1-\delta)M(Q_a)^{1-s_p}E(Q_a)^{s_p}$$
Then there exist $\delta'>0,\ 1>c>0$ such that: If $$\|u_0\|_{L^2}^{1-s_p}\|u_0\|_{\dot{H}_a^1}^{s_p}\leq\|Q_a\|_{L^2}^{1-s_p}\|Q_a\|_{\dot{H}^1}^{s_p}$$
then for all $t\in I$.
\begin{flalign*}
&(i)\|u(t)\|_{L^2}^{1-s_p}\|u(t)\|_{\dot{H}_a^1}^{s_p}\leq(1-\delta')\|Q_a\|_{L^2}^{1-s_p}
\|Q_a\|_{\dot{H}^1}^{s_p}&
\end{flalign*}
(ii)$\|u(t)\|_{\dot{H}_a^1}^2-\frac{p}{p+2}\|u\|_{L^{p+2}}^{p+2}\geq c\|u(t)\|_{\dot{H}_a^1}^2$

\end{lemma}
\begin{remark}
  By the coercivity and conserving mass, we may get $\|u\|_{H_a^1}$ uniformly bounded. Together with local well posedness,  the solution $u$ is global.
\end{remark}

\section{Proof of Theorem \ref{Main}}
\noindent

In this section, we turn to prove Theorem \ref{Main}. Let $u_0$ satisfies the hypotheses  of Theorem \ref{Main}, and let $u(t)$ be corresponding global-in-time solution to \eqref{INLS}. In particular, $u$ is uniformly bounded in $H_a^1$ and obeys the condition (i) of Lemma \ref{V-C} .

\subsection{Scattering Criterion}

\noindent

To show Theorem \ref{Main}, we first establish a scattering criterion by following the argument.
\begin{lemma}[Scattering Criterion]\label{S-C}
Suppose $u:\R_t\times\R^2\rightarrow \mathbb{C}$ is a radial solution to \eqref{INLS} such that
  \begin{align}\label{zz0}
    \|u\|_{L_t^\infty H_x^1(\R\times\R^2)}\leq E.
  \end{align}
There exist $\epsilon=\epsilon(E)>0$ and $R=R(E)>0$ such that
  \begin{align}\label{zz1}
   \liminf_{t\to \infty}\int_{|x|\leq R}|u(t,x)|^{2}dx\leq \epsilon ^{2},
  \end{align}
  and $u(t)$ satisfies

  \begin{align}\label{zz2}
  \int_{0}^{T}\int_{\R^2}|u(t,x)|^{p+2}dx\leq T^{\alpha},
  \end{align}
 where $0<\alpha<1$.
Then $u$ scatters forward in time.

\end{lemma}
\begin{proof}
   By standard continuity argument, Sobolev embedding and Strichartz, it suffices to show
  $\|u\|_{L_{t}^{\frac{4p}{3}}L_x^{4p}(\R\times\R^2)}<\infty$.

  By Duhamel formula and continuity argument, we need to prove
  $$\|e^{i(t-T_0)\Delta}u(T_0)\|_{L_{t}^{\frac{4p}{3}}L_x^{4p}([T_0,\infty)\times\R^2)}\ll 1$$
Noting that
\begin{align*}
  e^{i(t-T_{0})\Delta}u(T_{0})=e^{it\Delta}u_0-iF_{1}(t)-iF_{2}(t),
\end{align*}
where
\begin{align*}
  F_{j}(t):=\int_{I_{j}}e^{i(t-s)\Delta}(|x|^{-b}|u|^{p}u)(s)ds, ~j=1,2
\end{align*}
and $I_{1}=[0,T_{0}-\epsilon ^{-\theta}],\quad I_{2}=[T_{0}-\epsilon ^{-\theta},T_{0}].$

Let $T_0$ be large enough, we have
$$
  \|e^{it\Delta}u_0\|_{L_{t}^{\frac{4p}{3}}L_x^{4p}([T_0,\infty)\times\R^2)} \ll1.
$$
 It remains to show
\begin{align*}
  \|F_{j}(t)\|_{L_{t}^{\frac{4p}{3}}L_x^{4p}([T_0,\infty)\times\R^2)} \ll 1,\quad \text{~ for ~} j=1,2.
\end{align*}
\textbf{Estimation  of $F_{1}(t)$}: We may use the dispersive estimate, H\"older's inequality, thus
\begin{align*}
\left\|\int_{0}^{T_0-\epsilon^{-\theta}}e^{i(t-s)\Delta}|u|^puds\right\|_{L_x^\infty}\lesssim&
\int_{0}^{T_0-\epsilon^{-\theta}}|t-s|^{-1}\||u|^pu\|_{L_x^1}ds\\
\lesssim & \int_{0}^{T_0-\epsilon^{-\theta}}|t-s|^{-1}\left(\|u\|_{L_x^{p+2}}^{p}\|u\|_{L_x^{\frac{p+2}{2}}}\right)ds\\
\lesssim & \int_{0}^{T_0-\epsilon^{-\theta}}|t-s|^{-1}\|u\|_{L_x^{p+2}}^{p}ds\\
\lesssim & (T_0^{\alpha}\epsilon^{\theta})^{\frac{p}{p+2}}
\end{align*}
yielding
$$\left\|\int_{0}^{T_0-\epsilon^{-\theta}}e^{i(t-s)\Delta}|u|^puds\right\|_{L_{t,x}^\infty
(T_0,\infty)\times\R^2}\lesssim (T_0^{\alpha}\epsilon^{\theta})^{\frac{p}{p+2}}$$
On the other hand, we may
$$F_1(t)=e^{i(t-T_0+\epsilon^{-\theta})\Delta}u(T_0-\epsilon^{-\theta})-e^{it\Delta}u_0.$$
By Strichartz and $(\frac83,8)\in\Lambda_0$ we have
$$\|F_1(t)\|_{L_t^{\frac83}L_x^{8}}\lesssim 1,$$
Thus, by interpolation, we get
$$\|F_1(t)\|_{L_{t}^{\frac{4p}{3}}L_x^{4p}([T_0,\infty)\times\R^2)}\lesssim (T_0^{\alpha}\epsilon^{\theta})^{\frac{p-2}{p+2}}$$
\textbf{Estimation  of $F_{2}(t)$}: By Strichartz, Sobolev embedding and radial sobolev embedding, we get
\begin{align*}
 \|F_2(t)\|_{L_{t}^{\frac{4p}{3}}L_x^{4p}([T_0,\infty)\times\R^2)}\lesssim&\||u|^pu\|_{L_{t,x}^{\frac43}
 ([T_0-\epsilon^{-\theta}]\times\R^2)}+\|\nabla_A(|u|^pu)\|_{L_{t,x}^{\frac43}
 ([T_0-\epsilon^{-\theta}]\times\R^2)}\\
 \lesssim& \|u\|_{L_t^\infty L_x^2}\|u\|_{L_t^{\frac{4p}{3}}L_x^{4p}}^p+ \|\nabla u\|_{L_t^\infty L_x^2}\|u\|_{L_t^{\frac{4p}{3}}L_x^{4p}}^p+\||x|^{-1}u\|_{L_t^\infty L_x^2}\|u\|_{L_t^{\frac{4p}{3}}L_x^{4p}}^p\\
 \lesssim&\epsilon^{\frac{-3\theta}{4}}\|u\|_{L_t^\infty H_a^1}\|u\|_{L_t^\infty L_x^{4p}([T_0-\epsilon^{-\theta}]\times\R^2)}^p
\end{align*}
Let $T_0$ be large enough. By the assumption and identity $\pa_t|u|^2=-2\nabla\cdot Im(\bar{u}\nabla u)$, together with integration by parts and Cauchy-Schwartz, we duce
$$\left|\pa_t\int_{I_2}\chi_R|u|^2ds\right|\lesssim \frac{1}{R}$$
Thus, choosing $R\gg \epsilon^{-(2+\theta)}$, we find
$$\|\chi_Ru\|_{L_t^\infty L_x^2(I_2\times \R^2)}\lesssim \epsilon$$
Using radial Sobolev inequality and choosing $R$ large enough, we deduce
$$\|u\|_{L_t^\infty L_x^2(I_2\times\R^2)}\lesssim \epsilon$$
By interpolation, thus
\begin{align*}
\|u\|_{L_t^\infty L_x^{4p}(I_2\times\R^2)}\leq& \|u\|_{L_t^\infty L_x^2(I_2\times\R^2)}^{\frac{1}{4p}}\|u\|_{L_t^\infty L_x^{8p-2}(I_2\times\R^2)}^{\frac{4p-1}{4p}}\\
\leq&\epsilon^{\frac{1}{4p}}
\end{align*}
Then we may have
$$\|F_2(t)\|_{L_{t}^{\frac{4p}{3}}L_x^{4p}([T_0,\infty)\times\R^2)}\lesssim \epsilon^{\frac14-\frac{3\theta}{4}}$$

Choosing $\theta=\frac16, \epsilon=T^{-6(1+\gamma)}$, where $\gamma+\alpha<1$. Thus, we get
 $$\|e^{i(t-T_0)\Delta}u(T_0)\|_{L_t^{\frac{4p}{3}}L_x^{4p}([T_0,\infty)\times\R^2)}\ll 1$$
\end{proof}

\subsection{Virial/Morawetz identities}

\noindent

In this part, we recall the following general identity, which follows by computing directly using \eqref{INLS}.

\begin{lemma}\label{Mor-Iden}
  Let $u$ be the solution of \eqref{INLS} and $w(x)$ be a smooth function. We denote the Morawetz action $M_w(t)$ by
  $$M_w(t)~=~2\int_{\R^2} \nabla w(x)Im(\bar{u}\nabla u)(x)dx.$$
  Then we have
  \begin{align}
    \frac{d}{dt}M_w(t)~=~&-\int\Delta^2a(x)|u|^2dx+4\int \pa_{jk}a(x)Re(\pa_ku\pa_j\bar{u})dx\\
    &+4\int|u|^2\frac{ax}{|x|^4}\nabla wdx-\frac{2p}{p+2}\int|u|^{p+2}\Delta w dx.
  \end{align}
\end{lemma}
Let $R\gg 1$ to be chosen later. We take $w(x)$ to be a radial function satisfying

\begin{eqnarray}
w(x)=
\begin{cases}
|x|^2;& |x|\leq R\\
3R|x|;& |x|>2R,
\end{cases}
\end{eqnarray}
and when $R<|x|\leq 2R$, there holds
\begin{align*}
  \partial_{r}w\geq 0,\ \partial_{rr}w\geq 0\quad and \quad |\partial^{\alpha}w| \lesssim R|x|^{-|\alpha|+1}.
\end{align*}
Here $\partial_{r}$ denotes the radial derivative. Under these conditions, the matrix $(w_{jk})$ is non-negative.
It is easy to verify that
\begin{eqnarray*}
\begin{cases}
w_{jk}=2\delta_{jk},\quad \Delta w=4,\quad \Delta \Delta w=0,& |x|\leq R,\\
w_{jk}=\frac{3R}{|x|}[\delta_{jk}-\frac{x_{j}x_{k}}{|x|^2}],\quad \Delta w=\frac{3R}{|x|},\quad \Delta \Delta w=\frac{3R}{|x|^3},& |x|>2R.
\end{cases}
\end{eqnarray*}
Thus, we can divide $\frac{dM_w(t)}{dt}$ as follows:
\begin{align}\label{eq:main}
\frac{dM(t)}{dt}=&8\int_{|x|\leq R} |\nabla u|^2+a\frac{|u|^2}{|x|^2}-\frac{p}{p+2}|u|^{p+2}dx\\\label{eq:er1}
&+\int_{|x|>2R}4aR\frac{|u|^{2}}{|x|^3}dx+\int \frac{-12Rb}{(p+2)|x|}|u|^{p+2}dx+\int_{|x|>2R}\frac{12R}{|x|}|\not\!\nabla u|^2 dx\\\label{eq:er2}
&+\int_{R<|x|\leq 2R}4Re\bar{u}_{i}a_{ij}u_{j}+\mathcal{O}(\frac{R}{|x|}|u|^{p+2}+\frac{R}{|x|^3}|u|^2)dx,
\end{align}
where $\not\!\!\!\nabla$ denotes the angular derivation,
subscripts denote partial derivatives, and repeated indices are summed in this paper.

 Thus, we have
 \begin{align}\label{Mora:e}
 \int_{|x|\leq R}|\nabla u|^2+a\frac{|u|^2}{|x|^2}dx-\frac{p}{p+2}\int_{|x|\leq R}|u|^{p+2}dx \lesssim& \frac{dM(t)}{dt}+\frac{1}{R^{\alpha}}
\end{align}
where $\alpha=\min\{2,\frac{p}2\}$.

Let $\chi(x)$ be smooth function, denoted
\begin{eqnarray*}
\chi(x)=
\begin{cases}
1;\ \ |x|<\frac12\\
0;\ \ |x|>1
\end{cases}
\end{eqnarray*}
and $\chi_R(x)=\chi(\frac{x}{R})$.
\begin{lemma}[Coercivity on balls] \label{Cor}
There exists $R=R(\delta, M(u), Q)>0$ sufficiently large that
\begin{align}\label{Ball}
\sup_{t\in\R}\|\chi_{R}u\|_{L_x^2}^{1-s_p}\|\chi_{R}u\|_{\dot{H}_a^1}^{s_p}<
(1-\delta)\|Q\|_{L_x^2}^{1-s_p}\|Q\|_{\dot{H}^1}^{s_p}
\end{align}
 In particular, there exists $\delta'$ so that
 \begin{align}\label{Bianfen}
 \int|\nabla(\chi_{R}u)|^2dx-\frac{p}{p+2}\int|\chi_Ru|^{p+2}dx\geq\delta'\int|u|^{p+2}dx
 \end{align}
\end{lemma}
\begin{proof}
We will make use the following identity, which can be checked by direct computation:
\begin{align}
  \int\chi_R^2|\nabla u|^2=\int|\nabla(\chi_Ru)|^2+\chi_R\Delta(\chi_R)|u|^2dx
\end{align}
In particular, we have
$$\|\chi_Ru\|_{\dot{H}_a^1}^2\lesssim \|u\|_{\dot{H}_a^1}^2+\frac{1}{R^2}$$
Taking $R$ large enough, we obtain \eqref{Ball}. Then inequality \eqref{Bianfen} follows from \eqref{Ball}.
\end{proof}

By radial Sobolev inequality, Sobolev embedding and Lemma \ref{Cor}, we get
\begin{align}\label{Mora:e}
  \int_{|x|\leq \frac{R}{2}}|u|^{p+2}dx \lesssim& \frac{dM(t)}{dt}+\frac{1}{R^{\alpha}}
\end{align}
where $\alpha=\min\{2,b+\frac{p}2\}.$

\subsection{Proof of Theorem \ref{Main}}
\noindent

By the scattering criterion and H\"older inequality, Theorem \ref{Main} follows from Morawetz estimate (i.e. Proposition \ref{pro1}).

\begin{proposition}[Morawetz estimate]\label{pro1}
  Let $d=2$, $0<b<1,\ \ 0<p<\infty$ and $u$ be a solution to the focusing
\eqref{INLS} on the space-time slab $[0,T]\times \R^d$. Then
  \begin{align}\label{Main1}
  \int_{0}^{T}\int_{\R^d}|u|^{p+2}dxdt<T^{\beta_0}
  \end{align}
  where $\beta_0=\frac{1}{1+\alpha}<1.$

  Moreover, for any $R\gg1$, such that
  \begin{align}\label{Mian2}
  \liminf_{t\to\infty}\int_{|x|<R}|u|^{2}dx=0
  \end{align}
\end{proposition}
\begin{proof}
  By \eqref{Mora:e} and radial sobolev inequality, we may have
  $$\int_{\R^2}|u|^{p+2}dx \lesssim\frac{dM(t)}{dt}+\frac{1}{R^{\alpha}}$$
 Note that the uniform $H_A^1$-bounds for $u$, and the choice weight, we have
 $$\sup_{t\in\R}|M(t)|\lesssim R$$

 We now apply the fundamental theorem of calculus on an interval $[0,T]$, this yields
 $$\int_{0}^{T}\int_{\R^2} |u|^{p+2}dx \lesssim R+\frac{T}{R^{\alpha}}$$
Let $R=T^{\frac{1}{1+\alpha}},$ we get \eqref{Main1}.

By the same argument, we deduce
\begin{align*}
\int_{0}^{T}\int_{|x|\leq R}|u|^{p+2}dx\lesssim&R+\frac{T}{R}
\end{align*}
Let $R=T^{\frac{1}{1+\alpha}},$ we may have
$$\int_{0}^{T}\int_{|x|\leq R}|u|^{p+2}dx\leq T^{\frac{1}{1+\alpha}}$$
as desired. Therefore, we complete the proof of Theorem \ref{Main}.
\end{proof}

\begin{appendix}
\section{Interaction Morawetz type estimate}
\noindent

In this appendix, we will show that the global solution scatters for the defocusing case under the condition $a> \frac14$ via interaction Morawetz estimate. When $0<a<\frac14$, we don't know if the interaction Morawetz estimate is true. But for radial solution, we can remove the condition like focusing case.
\begin{proposition}\label{Morawetz}
   Let $u$ be an $H^1-$solution to
   \begin{align}\label{VNLS}
 \begin{cases}
   &i\pa_tu+\Delta u-V(x)u=|u|^{p}u,\ \ \ \  t\in\R,\ x\in\R^2\\
   &u(0,x)=u_0(x) \in H^1(\R^2)
 \end{cases}
 \end{align}
  on the spacetime slab $I\times\R^2$, and $a>\frac14$, then we have
  \begin{equation}\label{1.2}
\big\||\nabla|^{\frac12}(|u|^2)\big\|_{L^2(I;L^2(\R^2))}\leq
C\|u(t_0)\|_{L^2}^\frac32\sup_{t\in I}\|u(t)\|_{\dot
H^{1}}^\frac12,~t_0\in I,
\end{equation}
and hence
\begin{equation}\label{interac2}
\|u\|_{L_t^4(I;L_x^8(\R^2))}\leq C\|u(t_0)\|_{L^2}^\frac34\sup_{t\in
I}\|u(t)\|_{\dot H^{1}}^\frac14.
\end{equation}
\end{proposition}

To do this, we first show the local smoothing estimate as follows.
\begin{lemma}[Local smoothing estimate]\label{lem:locsm}
Let $u:~I\times\R^2$ be an $H^1$-solution to \eqref{VNLS} with $\lambda=1$, then we
have
\begin{align}\nonumber
(2a-\frac12)\int_I\int_{\R^2}\frac{|u(t,x)|^2}{|x|^3}\;dx\;dt\\\label{equ:locas}\lesssim\|u_0\|_{L_x^2}\|u\|_{L_t^\infty(I,\dot H^1(\R^2))}.
\end{align}
\end{lemma}

\begin{proof}
Define the Virial quantity
\begin{equation}
V(t):={\rm Im}\int_{\R^2}\bar{u}\partial_r u\;dx={\rm Im}\langle \partial_ru, u\rangle
\end{equation}
where $\partial_r u=\frac{x}{|x|}\cdot \nabla u$.
Deviating in $t$ and  by Leibniz rule, we obtain
\begin{align*}
\frac{d}{dt}V(t)=&{\rm Im}\int_{\R^2}\left(\bar{u}_t\partial_r
u +\bar{u}\partial_r u_t\right)dx\\
\triangleq&I_1+I_2.
\end{align*}

{\bf The contribution of $I_1$:} Using \eqref{VNLS} and
integration by parts, we get
\begin{align*}
I_1=&-{\rm Re}\int_{\R^2}\Delta\bar{u}\partial_r u\;dx+a{\rm
Re}\int_{\R^2}\frac{\bar{u}}{|x|^2}\partial_r u\;dx+{\rm Re}\int_{\R^2}|u|^{p-1}\bar{u}\partial_r u\;dx\\
=&\frac12\int_{\R^2}\frac{|\nabla
u|^2}{|x|}\;dx+\int_{\R^2}\frac{|\nabla u|^2-|\partial_r u|^2}{|x|}\;dx+\frac{a}{2}\int_{\R^2}\frac{|u|^2}{|x|^3}\;dx-\frac1{p+1}\int_{\R^2}\frac{|u|^{p+1}}{|x|}\;dx.
\end{align*}

{\bf The contribution of $I_2$:} From \eqref{VNLS} and integration
by part, we estimate
\begin{align*}
I_2=&{\rm Re}\int_{\R^2}\bar{u}\frac{x}{|x|}\cdot\nabla\Delta
u\;dx-a{\rm
Re}\int_{\R^2}\bar{u}\frac{x}{|x|}\cdot\nabla\Big(\frac{u}{|x|^2}\Big)\;dx-{\rm
Re}\int_{\R^2}\bar{u}\frac{x}{|x|}\cdot\nabla(|u|^{p-1}u)\;dx\\
=&\frac12\int_{\R^2}\frac{|\nabla
u|^2}{|x|}\;dx+\int_{\R^2}\frac{|\nabla u|^2-|\partial_r u|^2}{|x|}\;dx+\Big(\frac{3}{2}a-\frac12\Big)\int_{\R^2}\frac{|u|^2}{|x|^3}\;dx+\frac{p}{p+1}\int_{\R^2}\frac{|u|^{p+1}}{|x|}\;dx.
\end{align*}

Hence,
\begin{align*}
\frac{d}{dt}V(t)=&2\int_{\R^2}\frac{|\nabla u|^2-|\partial_r u|^2}{|x|}\;dx+(2a-\frac12)\int_{\R^2}\frac{|u|^2}{|x|^3}\;dx+\frac{p-1}{p+1}\int_{\R^2}\frac{|u|^{p+1}}{|x|}\;dx\\
\geq&(2a-\frac12)\int_{\R^2}\frac{|u|^2}{|x|^3}\;dx.
\end{align*}
Integrating on  time interval $I$ implies
\begin{align}\label{gj}
&(2a-\frac12)\int_I\int_{\R^2}\frac{|u(t,x)|^2}{|x|^3}\;dx\;dt
\leq2\sup_{t\in I }|V(t)|
\lesssim \|u_0\|_{L_x^2}\|u\|_{L_t^\infty(I,\dot H^1(\R^2))}.
\end{align}
\end{proof}
\begin{proof}[{\bf The proof of Proposition \ref{Morawetz}:}]
We consider the NLS equation in the form of
\begin{equation}\label{NLS}
i\partial_tu+\Delta u=gu
\end{equation}
where $g=g(\rho,|x|)=|u|^p+V(x)$ with $V(x)=\frac{a}{|x|^2}$.  Define the pseudo-stress energy tensors associated with Schr\"odinger equation for $j=1,2$
\begin{equation}
\begin{split}
T_{00}&=\tfrac12|u|^2,\\
T_{0j}&=\mathrm{Im}(\bar
u \pa_j u),\\
T_{jk}&=2\mathrm{Re}(\pa_j u
\overline{\pa_k u})-\tfrac12\delta_{jk}\Delta(|u|^2).
\end{split}
\end{equation}
We have by  \cite{CCL}
\begin{equation}\label{Local Conservation}
\begin{split}
\partial_t T_{00}+\partial_j T_{0j}&=0,\\
\partial_t T_{0j}+\partial_k T_{jk}&=-\rho\partial_j g.
\end{split}
\end{equation}

By the density argument, we may assume sufficient smoothness and
decay at infinity of the solutions to the calculation and in
particular to the integrations by parts. Let $h=|x|$. The starting
point is the auxiliary quantity
\begin{equation*}
J=\tfrac12\langle|u|^2, h\ast |u|^2\rangle=2\langle T_{00}, h\ast
T_{00}\rangle.
\end{equation*}
Define the quadratic Morawetz quantity $M=\tfrac14\partial_t J$.
Hence we can precisely rewrite
\begin{equation}\label{3.1}
M=-\tfrac12\langle\partial_jT_{0j}, h\ast
T_{00}\rangle-\tfrac12\langle T_{00}, h\ast
\partial_jT_{0j} \rangle=-\langle T_{00}, \partial_j h\ast
T_{0j} \rangle.
\end{equation}
By \eqref{Local Conservation} and integration by parts, we have
\begin{equation*}
\begin{split}
\partial_tM=&\langle\partial_kT_{0k}, \partial_j h\ast T_{0j} \rangle-\langle T_{00},
\partial_j h\ast\partial_t T_{0j} \rangle\\=&-\sum_{j,k=1}^n\langle T_{0k}, \partial_{jk} h\ast T_{0j} \rangle+\langle T_{00},
\partial_{jk} h\ast T_{jk} \rangle+\frac12\langle \rho,
\partial_j h\ast(\rho\partial_j g) \rangle.
\end{split}
\end{equation*}
A simple computation gives
\begin{equation}
\begin{split}
\sum_{j,k=1}^n\langle T_{0k},  \partial_{jk} h\ast T_{0j}
\rangle&=\big\langle \mathrm{Im}(\bar u\nabla u), \nabla^2 h\ast
\mathrm{Im}(\bar u\nabla u) \big\rangle\\&=\big\langle \bar u\nabla
u, \nabla^2 h\ast \bar u\nabla u \rangle-\langle \mathrm{Re}(\bar
u\nabla u), \nabla^2 h\ast \mathrm{Re}(\bar u\nabla u) \big\rangle.
\end{split}
\end{equation}
Therefore it yields that
\begin{equation*}
\begin{split}
\partial_tM=&\big\langle \mathrm{Re}(\bar u\nabla
u), \nabla^2 h\ast \mathrm{Re}(\bar u\nabla u)
\big\rangle-\big\langle \bar u\nabla u, \nabla^2 h\ast \bar u\nabla
u \big\rangle\\&+\Big\langle \bar uu,
\partial_{jk} h\ast \big(\mathrm{Re}(\pa_j u \overline{\pa_k
u})-\tfrac14\delta_{jk}\Delta(|u|^2)\big)
\Big\rangle+\frac12\big\langle \rho,
\partial_j h\ast(\rho\partial_j g) \big\rangle.
\end{split}
\end{equation*}
Note
\begin{equation*}
\begin{split}
-\big\langle \bar uu,
\partial_{jk} h\ast\delta_{jk}\Delta(|u|^2) \big\rangle=\big\langle \nabla (|u|^2), \Delta h\ast
\nabla(|u|^2) \big\rangle,
\end{split}
\end{equation*}
and $$\big\langle \mathrm{Re}(\bar u\nabla u), \nabla^2 h\ast
\mathrm{Re}(\bar u\nabla u) \big\rangle=\frac14\big\langle
\nabla(|u|^2),\nabla^2h\ast\nabla(|u|^2)\big\rangle=\frac14\big\langle
\nabla(|u|^2),\Delta h\ast\nabla(|u|^2)\big\rangle,$$
 we write
\begin{equation}\label{Morawetz equality}
\begin{split}
\partial_tM=\tfrac12\langle \nabla \rho, \Delta h\ast\nabla\rho \rangle+R+\tfrac12\big\langle \rho,
\partial_j h\ast(\rho\partial_j g) \big\rangle.
\end{split}
\end{equation}
Here $R$ is given by
\begin{equation*}\label{3.4}
\begin{split}
R&=\big\langle \bar uu, \nabla^2 h\ast (\nabla\bar u \nabla u)
\big\rangle-\big\langle \bar u\nabla u, \nabla^2 h\ast \bar u\nabla
u \big\rangle\\&=\tfrac12\int \Big(\bar u(x)\nabla \bar u(y)-\bar
u(y)\nabla\bar u(x)\Big)\nabla^2h(x-y)\Big(u(x)\nabla
u(y)-u(y)\nabla u(x)\Big)\mathrm{d}x\mathrm{d}y.
\end{split}
\end{equation*}
Since the Hessian of $h$ is positive definite, we have $$R\geq 0.$$
Integrating over time in an interval $[t_1, t_2]\subset I$ yields
\begin{equation*}
\begin{split}
\int_{t_1}^{t_2}\Big\{\langle \nabla \rho, \Delta h\ast\nabla\rho
\rangle&+\langle \rho,
\partial_j h\ast(\rho\partial_j g) \rangle\Big\}\mathrm{d}t\leq-2\langle T_{00}, \partial_j h\ast
T_{0j} \rangle\big|_{t=t_1}^{t=t_2}.
\end{split}
\end{equation*}

From now on, we choose $h(x)=|x|$. It is easy to see that
\begin{equation*}
\Big|\mathrm{Im}\int_{\R^2}\int_{\R^{2}}|u(x)|^2\frac{x-y}{|x-y|}\bar
u(y)\nabla u(y)dxdy\Big|\leq C\sup_{t\in
I}\|u(t)\|^3_{L^2}\|u(t)\|_{\dot H^{1}}.
\end{equation*}
Hence,
\begin{equation}\label{Morawetz inequality}
\int_{I}\big\langle \rho,
\partial_j h\ast(\rho\partial_j g) \big\rangle dt+\big\||\nabla|^{\frac{1}2}(|u|^2)\big\|_{L^2(I;L^2(\R^2))}^2\leq C\sup_{t\in
I}\|u(t)\|_{L^2}^3\|u(t)\|_{\dot H^{1}}.
\end{equation}

Now we consider the term
\begin{equation*}
\begin{split}
P&:=\big\langle \rho, \nabla h\ast (\rho\nabla g) \big\rangle.
\end{split}
\end{equation*}
Consider $g(\rho,|x|)=\rho^{(p-1)/2}+V(x)$ with
$V(x)=\frac{a}{|x|^2}$, then we can write $P=P_1+P_2$ where
\begin{equation}\label{3.5}
\begin{split}
P_1= \big\langle\rho, \nabla h\ast \big(\rho\nabla
(\rho^{(p-1)/2})\big)\big\rangle=\frac{p-1}{p+1}\big\langle\rho,
\Delta h\ast \rho^{(p+1)/2}\big\rangle\geq 0
\end{split}
\end{equation}
and
\begin{equation}\label{P2}
\begin{split}
P_2=& \iint\rho(x)\nabla h(x-y)\rho(y)\nabla
\big(V(y)\big)\mathrm{d}x\mathrm{d}y\\
=&2a\iint|u(x)|^2\frac{(x-y)}{|x-y|}\cdot
\frac{y}{|y|^{4}}|u(y)|^2\;dx\;dy.
\end{split}
\end{equation}
Using \eqref{equ:locas}, we get
\begin{equation*}
\begin{split}
\int_{t_1}^{t_2}|P_2(t)|\;dt\leq
2a\|u_0\|_{L_x^2}^2\int_{t_1}^{t_2}\int_{\R^2}\frac{|u(t,x)|^2}{|x|^3}\;dx\;dt
\leq\frac{8a}{4a-1}\sup_{t\in
I}\|u_0\|_{L_x^2}^3\|u(t)\|_{\dot H^{1}}.
\end{split}
\end{equation*}
This together with \eqref{3.5} and \eqref{Morawetz inequality}
yields \eqref{1.2}. The result follows.
\end{proof}
\end{appendix}

\begin{center}

\end{center}
\end{CJK*}
 \end{document}